\documentclass{amsart}     
\usepackage{amssymb}

\newcommand{\C}{\mathbb{C}}
\newcommand{\D}{\mathbb{D}}

\newcommand{\R}{\mathbb{R}}
\newcommand{\N}{\mathbb{N}}

\newcommand{\dom}{\mbox{dom}}
\newcommand{\ran}{\mbox{ran}}

\renewcommand{\Re}{\mbox{Re}}

\newtheorem{theorem}{Theorem}[section]
\newtheorem{lemma}[theorem]{Lemma}

\theoremstyle{definition}

\theoremstyle{theorem}

\theoremstyle{theorem}

\theoremstyle{theorem}

\theoremstyle{theorem}

\theoremstyle{definition}

\theoremstyle{theorem}

\numberwithin{equation}{section}

\begin{document}

\title{An effective Carath\'eodory Theorem}

\author{Timothy H. McNicholl}
\address{Department of Mathematics\\
              Lamar University\\
              Beaumont, Texas 77710 USA \\
               }
\email{timothy.h.mcnicholl@gmail.com}

\begin{abstract}
By means of the property of effective local connectivity, the computability of finding the Carath\'eodory extension of a conformal map of a 
Jordan domain onto the unit disk is demonstrated.
\end{abstract}
\keywords{Computable analysis, constructive analysis, complex analysis, conformal mapping, effective local connectivity
}
\subjclass[2010]{03F60, 30C20, 30C30, 30C85}

\maketitle

\section{Introduction}

In 1851, B. Riemann proved his powerful Riemann Mapping Theorem which states
that every simply connected subset of the plane is conformally equivalent to the 
unit disk, $\D$.  E. Bishop gave a constructive proof of this theorem in 1967 (now in \cite{Bishop.Bridges.1985}), and in 1999 P. Hertling proved a uniformly effective version of this theorem \cite{Hertling.1999}.

In classical complex analysis, the next step up from the Riemann Mapping Theorem is
the Carath\'eodory Theorem which states that a conformal map of a simply connected
Jordan domain onto the unit disk can be extended to a homeomorphism of the closure
of the domain with the closure of the disk.  Such an extension is called the 
{\it boundary extension} or {\it Carath\'eodory extension}.  This result is 
the basis for demonstrating the existence of solutions to Dirichlet problems on 
simply connected Jordan domains.

Here, we will extend Hertling's work by proving a uniformly effective version of the Carath\'eodory Theorem.  That is, roughly speaking, we show that from sufficiently good approximations to a 
parameterization of a Jordan curve $\gamma$, we can compute arbitrarily good approximations
of a conformal map of the interior of $\gamma$ onto the unit disk as well as arbitarily 
good approximations of its boundary extension.  This will be achieved by means of \it effective local connectivity\rm.
This concept, which first appeared in \cite{Miller.2004}, is the subject of a previous 
investigation by Daniel and McNicholl \cite{Daniel.McNicholl.2009} and also plays an important role in 
\cite{Brattka.2008}.  The resulting construction of the Carath\'eodory extension can be considered as a 
new proof of this result.  

The paper is organized as follows.  In Sections \ref{sec:BACK} and \ref{sec:COMP}, we summarize background
information from complex analysis and computable analysis.  In Section 
\ref{sec:MAIN}, we head straight for the proof of the main theorem.  

\section{Background from complex analysis}\label{sec:BACK}

Most of our terminology and notation from complex analysis can be found in \cite{Conway.1978} and \cite{Conway.1995}.

We will use $C[0,1]$ to denote the set of continuous functions from $[0,1]$ into 
$\C$.

When $X \subseteq \C$, define 
\[
D_\epsilon(X) = \bigcup_{p \in X} D_\epsilon(p).
\]
Let $\parallel\ \parallel_\infty$ denote the $L^\infty$ norm on $C[0,1]$.
Hence, if $\parallel f - g\parallel_\infty < \epsilon$, then $\ran(f) \subseteq D_\epsilon(\ran(g))$.

An $f \in C[0,1]$ is \emph{polygonal} if there is a partition of 
$[0,1]$, $0 = a_0 < a_1 < \ldots < a_{k-1} < a_k = 1$, such that $f$ is linear on 
each of $[a_j, a_{j+1}]$.  $f$ is \emph{rational polygonal} if each $a_j$ is a rational
number and each $f(a_j)$ is a rational point.  Such a function is completely determined
by the tuple $(a_0, f(a_0), \ldots, a_k, f(a_k))$.  

We say that an arc $A_2$ \emph{extends} an arc $A_1$ if $A_2 \supset A_1$ and 
$A_2, A_1$ have an endpoint in common.

A \emph{parameterization} of an arc $A$ is a homeomorphism of $[0,1]$ with $A$.  A parameterization of a Jordan curve $J$ is a homeomorphism of $\partial \D$ with
$J$.  We will follow the usual convention of identifying an arc or Jordan curve with any of its parameterizations.

A \emph{domain} is an open, connected subset of the plane.  A domain is \emph{Jordan} if it is bounded by Jordan curves.

A function $u$ on a domain $D$ is \emph{harmonic} if 
\[
\frac{\partial^2 u}{\partial x^2}(z) + \frac{\partial^2 u}{\partial y^2} (z) = 0
\]
for all $z \in D$.  

Suppose $D$ is a Jordan domain and that $f$ is a bounded piecewise continuous
function on its boundary.  The resulting \emph{Dirichlet problem} is to find a harmonic
function $u$ on $D$ with the property that 
\[
\lim_{z \rightarrow \zeta} u(z) = f(\zeta)
\]
for all $\zeta \in \partial D$ at which $f$ is continuous.  Solutions to Dirichlet problems
always exist and are unique.  The function $f$ is said to provide the \emph{boundary data} for this 
problem.

\section{Background from computable analysis and computability}\label{sec:COMP}

An informal summary of the fundamentals of Type-Two Effectivity appears in 
\cite{Daniel.McNicholl.2009}.  We identify $\C$ with $\R^2$.  Hence, we can also use the naming systems
in \cite{Daniel.McNicholl.2009} for $\C$ and its hyperspaces.  
We shall identify objects with their names wherever this results in simplicity of exposition while not 
creating misconceptions.  

While we identify arcs and Jordan curves with their parameterizations, such curves will always be named by names of their parameterizations.

We also refer the reader to \cite{Daniel.McNicholl.2009} for the definitions of \it local connectivity witness\rm\ and 
\it uniform local arcwise connectivity witness\rm\ and related theorems.  In addition to 
these theorems, we will need the following.

\begin{lemma}
From a name of a homeomorphism $f$ of $\partial \D$ with a Jordan curve $J$, we can compute a local connectivity witness for $J$.  
\end{lemma}

\begin{proof}
It follows by essentially the same argument as in the proof of Theorem 6.2.7 of \cite{Weihrauch.2000} that we can compute, uniformly in the given data, a modulus of continuity for $f$, $m$.
  Compute $f^{-1}$ and a 
modulus of continuity for $f^{-1}$, $m_1$.  Let $h = m_1 \circ m$.

We claim that $h$ is a local connectivity witness for $J$.  
For, let $k \in \N$, and let $w_1 \in J$.  Let $z_1$ be the unique preimage of 
$w_1$ under $f$.  Let $I$ be the interval $(z_1 - 2^{-m(k)}, z_1 + 2^{-m(k)})$.
Let $C = f[I]$.  
We claim that 
\[
J \cap D_{2^{-h(k)}}(w_1) \subseteq C \subseteq D_{2^{-k}}(w_1).
\]
For, let $w_2 \in C$.  Let $z_2$ be the unique preimage of $w_2$ under $f$.
Then, $z_2 \in I$, and so $|z_1 - z_2| < 2^{-m(k)}$.  Hence, $|w_2 - w_1| < 2^{-k}$.
Thus, $C \subseteq S_{2^{-k}}(w_1)$.  

Now, suppose $w_2 \in J \cap D_{2^{-h(k)}}(w_1)$.  Again, let $z_2$ be the unique preimage of $w_2$ under $f$.  Then, $|z_1 - z_2| < 2^{-m(k)}$.  
Hence, $z_2 \in I$.  Thus, $w_2 \in C$.  
\end{proof}

\begin{lemma}\label{lm:DIRICHLET.DISK}
Given names of arcs $\gamma_1, \ldots, \gamma_n$ such that $\partial \D = \gamma_1 + \ldots + \gamma_n$,
and given names of continuous real-valued functions $f_1, \ldots, f_n$ such that $\gamma_j = \dom(f_j)$, we can compute a name of the harmonic
function $u$ on $\D$ defined by the boundary data
\[
f(\zeta) = \left\{ \begin{array}{cc}
					f_j(\zeta) & \zeta \in \gamma_j, \zeta \neq \gamma_j(0), \gamma_j(1)\\
					\max_j \max f_j & \mbox{otherwise}.\\
					\end{array}
					\right.
\]
In addition we can compute the extension of $u$ to $\overline{\D}$ except
at the endpoints of the arcs $\gamma_1, \ldots, \gamma_n$.
\end{lemma}

\begin{proof}
Let $u$ be the solution to the resulting Dirichlet problem on $\D$.  For $z \in \D$, we use the Poisson Integral Formula 
\begin{eqnarray}
u(z) = \frac{1}{2\pi}\int_0^{2\pi} u(e^{i\theta}) \frac{1 - |z|^2}{|e^{i\theta} - z|^2} d\theta\ z \in \D.\label{eqn:DIRICHLET.1}
\end{eqnarray}
(See, for example, Theorem I.1.3 of \cite{Garnett.Marshall.2005}.)
In the case under consideration, we have
\[
u(z) = \sum_j \frac{1}{2\pi} \int_{\gamma_j} f_j(\zeta) \frac{1 - |z|^2}{|\zeta - z|^2}|d\zeta|.
\]
Since integration is a computable operator, this shows we can compute $u$ on
$\D$.  

Since we are given $f_1, \ldots, f_n$, it might seem immediate that we can now compute the extension
of $u$ to $\overline{\D}$ except at the endpoints of $\gamma_1, \ldots, \gamma_n$.  However, it is not possible to determine from a name of a point $z \in \overline{\D}$ if $z \in \partial \D$.  
To see what the difficulty is, and to lead the way towards its solution, we delve
a little more deeply into the formalism.  Suppose we are given a name of a $z \in \overline{\D}$, $p$.  As we read $p$, it may be that at some point we find a subbasic neighborhood $R$ whose closure is contained in $\D$.  In this case, we can just use equation (\ref{eqn:DIRICHLET.1}).  However, if we 
keep finding subbasic neighborhoods that intersect $\partial \D$, 
then at some point we must commit to an estimate of $u(z)$.  If we
guess $z \in \partial \D$, then later this guess and this resulting estimate may
turn out to be incorrect.  We face a similar problem if we guess $z \in \D$.  The 
heart of the matter then is to estimate the value of $u(\zeta)$ when $\zeta$ is
near $z$ and in $\D$.  This can be done by effectivizing one of the usual proofs 
that $\lim_{\zeta \rightarrow z} u(\zeta) = f(z)$ when $z$ is between the endpoints of 
a $\gamma_j$.  To begin, fix rational numbers $\alpha \in [-\pi, \pi]$, 
$2\pi > \delta > 0$, and $0 < \rho < 1$.  Let  
\[
S(\rho, \delta, \alpha) =_{df} \{re^{i\theta}\ |\ \rho < r \leq 1\ \wedge\ |\theta - \alpha| < \delta/2\}.
\]
We now write the solution to the Dirichlet problem on the disk in a slightly different way
that considered previously in this proof.  To this end, 
let $P_r(\theta)$ be the \emph{Poisson kernel}, 
\[
Re\left(\frac{1 + re^{i\theta}}{1 - re^{i\theta}}\right).
\]
It is fairly well-known that if $\zeta \in \partial \D$ and $z \in \D$, then 
\[
\frac{1 - |z|^2}{|\zeta - z|^2} = \Re\left( \frac{\zeta + z}{\zeta - z} \right).
\]
At the same time, if $z = re^{i \theta}$ and $\zeta = e^{i \theta'}$, then 
\[
Re\left(\frac{\zeta + z}{\zeta - z} \right) = P_r(\theta - \theta').
\]
Let $f$ be a function on $\partial \D$ such that for each arc $\gamma_j$ 
$f(\zeta) = f_j(\zeta)$ whenever $\zeta$ is a point in $\gamma_j$ besides one
of its endpoints.  It then follows that when $0 < r < 1$,  
\[
u(re^{i \theta_1}) = \frac{1}{2\pi} \int_{-\pi}^\pi f(e^{i\theta}) P(\theta_1 - \theta) d\theta.
\]
We can compute a rational number $M$ such that 
\[
M > \max_k \max \ran(f_k).
\]
Suppose $\alpha$ is such that for some $j$, $e^{i \alpha}$ is in $\gamma_j$ but is not an endpoint.  From the given data, we can enumerate all such $\alpha, j$.  We cycle through all such $\alpha, j$ as we scan the name of $z$.  Fix a rational number $\epsilon > 0$.  From $\epsilon$ and the given data, one can compute a rational $\delta_{\alpha,\epsilon} > 0$ such that the arc
\[
\{e^{i \theta}\ :\ |\theta - \alpha| \leq \delta_{\alpha,\epsilon}\}
\]
is contained in $\ran(\gamma_j)$ and such that
\begin{eqnarray*}
\frac{\epsilon}{3} & > & \max\{|f_j(e^{i\theta}) - f_j(e^{i\alpha})|\ :\ |\theta - \alpha| \leq \delta_{\alpha,\epsilon}\}.
\end{eqnarray*}
We claim we can then compute $\rho_{\alpha, \epsilon}$ such that
\begin{eqnarray*}
\frac{\epsilon}{3M} & > & \max\{P_r(\theta)\ :\ |\theta| \geq \frac{1}{2}\delta\ \wedge\ \rho_{\alpha, \epsilon} \leq r \leq 1\}.
\end{eqnarray*}
We postpone the computation of $\rho_{\alpha, \epsilon}$ so that we can reveal our intent.  Namely, we claim that $|u(\zeta) - f(e^{i\alpha})| \leq \epsilon$ when $\zeta \in S(\rho_{\alpha, \epsilon}, \delta_{\alpha, \epsilon}, \alpha)$.
For, let $\zeta \in S(\rho_{\alpha, \epsilon}, \delta_{\alpha, \epsilon}, \alpha)$, and 
write $\zeta$ as $re^{i \theta_1}$.  Hence, $\rho < r \leq 1$, and we can 
choose $\theta_1$ so that $|\theta_1 - \alpha | < \delta_{\alpha, \epsilon} /2$.
If $r =1$, then there is nothing more to do.  So, suppose $r < 1$.
For convenience, abbreviate $\delta_{\alpha, \epsilon}$ by $\delta$.  
It then follows that 
\begin{eqnarray*}
|u(\zeta) - f(e^{i \alpha})| & \leq & \frac{1}{2\pi} \int_{|\theta - \alpha| \geq \delta} |f(e^{i \theta}) - f(e^{i \alpha})| P(\theta_1 - \theta) d\theta\\
 & & + \frac{1}{2\pi} \int_{|\theta - \alpha| < \delta}  |f(e^{i \theta}) - f(e^{i \alpha})| P(\theta_1 - \theta) d\theta.
\end{eqnarray*}
Suppose $|\theta - \alpha| \geq \delta$.  Since $|\theta_1 - \alpha| < \delta/2$, 
$|\theta_1 - \theta| \geq \delta/2$ and so 
$P_r(\theta_1 - \theta) < \epsilon /3M$.  Hence, the first term in the preceding sum
is at most $2\epsilon /3$.  At the same time, by our choice of $\delta_{\alpha, \epsilon}$, 
it follows that the second term is no larger than 
\[
\frac{1}{2\pi} \int_{-\pi}^\pi \frac{\epsilon}{3} P_r(\theta_1 - \theta) d\theta \leq \frac{\epsilon}{3}.
\]
Hence, $|u(\zeta) - f(e^{i \alpha})| \leq \epsilon$.

Hence, as we scan the name of $z$, if we encounter a rational rectangle $R$ and $\epsilon, \alpha$ such that $S(\rho_{\alpha, \epsilon}, \delta_{\alpha, \epsilon})$ contains $R$, then we can list any subbasic neighborhood that contains $[f(e^{i \alpha}) - \epsilon, f(e^{i\alpha}) + \epsilon]$.  It follows that if we
only read neighborhoods that intersect the boundary, then we will write a name 
of $f(z)$ on the output tape.

We conclude by showing how to compute $\rho_{\alpha, \epsilon}$.  Let $\delta' = \frac{1}{2} \delta$.  The key inequality is
\[
P_r(\theta) \leq P_r(\delta')
\]
when $\delta' \leq |\theta| \leq \pi$.  This is justified by Proposition 2.3.(c) on page 257 of \cite{Conway.1978}.  Since $e^{i \delta'} \neq 1$, 
$P_1(\delta')$ is defined, and in fact is $0$.  We can compute $P_r(\delta')$ as a 
function of $r$ on $[0,1]$.  We can thus compute 
$\rho_{\alpha, \epsilon}$ as required.
\end{proof}

\section{Proof of the main theorem}\label{sec:MAIN}

\begin{theorem}[\bf Effective Carath\'eodory Theorem]\label{thm:CARATH}
From a name of a parameterization $f$ of a Jordan curve $J$, and a name of a conformal map $\phi$ of the interior of $J$ onto $\D$, it is possible to 
compute a name of the Carath\'eodory extension of $\phi$.
\end{theorem}

\begin{proof}
It follows from the main theorem of \cite{Gordon.Julian.Mines.Richman.1975} that 
we can compute, uniformly from the given data, a name of the interior of $J$.  It then follows that
from the given data, we can compute a rational point $z_0$ in the interior of $J$.  We can also compute
a uniform local arcwise connectivity witness for $J$, $h$.  We can assume $h$ is increasing.
Finally, from the given data, we can uniformly compute a sequence of rational polygonal
Jordan curves $\{P_t\}_{t \in \N}$ such that 
$\parallel P_t - J \parallel_\infty < 2^{-t}$.  See, for example, Lemma 6.1.10 of \cite{Weihrauch.2000}.  Hence, 
$J \subseteq D_{2^{-t}}(P_t)$.  Let $D$ denote the interior of $J$.

Let us allow $\phi$ to denote the Carath\'eodory extension of $\phi$.  
The outline of our proof is as follows.  We first compute the restriction of $\phi$ to 
the boundary of $D$.  It then follows that we can compute $\phi^{-1}$ on $\partial \D$.  It then 
follows by applying Lemma \ref{lm:DIRICHLET.DISK} to the real and imaginary parts of $\phi^{-1}$ that we can compute
$\phi^{-1}$ on the closure of $\D$ and hence $\phi$ on the closure of $D$.  

To compute $\phi$ on the boundary of $D$, we appeal to the Principle of Type Conversion.   Namely, we suppose are additionally
given a name of a $\zeta_0 \in J$ and compute $\phi(\zeta_0)$.  The key to this is to compute an arc $Q$ from $z_0$ to $\zeta_0$ such that $Q \subseteq \overline{D}$ and $Q \cap J = \{\zeta_0\}$.  
We do this by computing a sequence of rational polygonal arcs 
$\{Q_t\}_{t \in \N}$ as follows.

To compute $Q_0$, we first compute a rational point $e_0 \neq z_0$ in the interior of $J$ whose distance from $\zeta_0$ is less than $2^{-h(0)}$.  Now, $D$ is an 
open connected set.  It follows that there is a rational polygonal arc from 
$z_0$ to $e_0$ which is contained in the interior of $J$.   Such an arc can now 
be discovered through a search procedure.

We now describe how we compute $Q_{t+1}$.  This is to be a rational polygonal 
arc with the following properties. 
\renewcommand{\theenumi}{(\arabic{enumi})}
\begin{enumerate}
	\item $Q_{t+1}$ is contained in the interior of $J$.\label{prop1}
	
	\item $Q_{t+1}$ extends $Q_t$ and has $z_0$ as a common endpoint.\label{prop2}
	
	\item The other endpoint of $Q_{t+1}$, $e_{t+1}$, is such that 
	$|e_{t+1} - \zeta_0| < 2^{-h(t+1) - 1}$.
	
	\item The points in $Q_{t+1}$ that are not in $Q_t$ are contained in 
	the disk of radius $2^{-t + 1}$ about $e_t$.\label{prop4}
	
\end{enumerate}

By way of induction, suppose $Q_t$ is a rational polygonal arc from 
$z_0$ to a point labelled $e_t$ and that 
$|e_t - \zeta_0| < 2^{-h(t)-1}$.  Suppose also that $Q_t$ is contained in the interior of $J$.
We first compute a rational point $e_{t+1}$ in the interior of $J$ and an integer $r_{t+1}$ such that 
$|e_{t+1} - \zeta_0| < 2^{-h(t+1) - 1}$, 
$e_{t+1} \in \C - \overline{D_{2^{-r_{t+1}}}(P_{r_{t+1}})}$, 
 $|e_{t+1} - e_t| < 2^{-h(t)-1}$, and $e_{t+1} \not \in Q_t$.
 
We summarize the key claim at this point in the proof by the following Lemma.

\begin{lemma}
$e_{t+1}$ and $e_t$ are in the same connected component
 of 
 \[
D \cap D_{2^{-t} + 2^{-h(t)-1}}(e_t).
 \]
 \end{lemma}
 
 \begin{proof}
 For convenience, let $\epsilon = 2^{-t}$,  
 $\delta = 2^{-h(t)}$, and $B = D_{\epsilon + \delta/2}(e_t)$.\footnote{I wish to express here my gratitude to my colleague Dr. Dale Daniel for allowing me to use this proof, which is entirely of his own creation, in this paper.}
 
 By way of contradiction, suppose $e_{t+1}$ and $e_t$ are not in the 
 same connected component of $D \cap B$.
 Let $E$ and $E'$ be the distinct components that 
 contain $e_t$ and $e_{t+1}$ respectively.  Since $B$ is convex,
 there is a line segment $l$ having $e_t$ and $e_{t+1}$ as endpoints
 and such that $l \subset B$.  The length of $l$ is of course 
 less that $\delta/2$.  
Let $P = J \cap l \cap \partial E$, and let 
$P' = J \cap I \cap \partial E'$.  Each of $P$ and $P'$ is closed and 
therefore compact as a subset of the plane.  Therefore, 
there exists $p \in P$ and $p' \in P'$ such that 
$|p - p'| = d(P, P')$.  Hence, 
$d(P, P') < \delta /2$.  

For points, $a$ and $b$ of $J$, let $l(a,b)$ denote the 
minimum diameter of an arc in $J$ having $a$ and $b$ as endpoints.
Hence, $|a - b| \leq l(a,b)$.  

Now, let $A$ denote an arc in $J$ such that the endpoints of 
$A$ are $p$ and $p'$ and so that the diameter of $A$ is 
$l(p,p')$.  Note that $l(p, p') < \epsilon$.
A straightforward connectivity argument shows that $A$ 
intersects $\partial B$.  (If not, then there is a 
piecewise linear arc with one endpoint in $A$ and the other in 
$E'$- a contradiction.)

$A \cap \partial B$ is a compact subset of $A$.  Hence, there is a 
point $q \in A \cap B$ such that the diameter of the subarc of $A$
with endpoints $p$ and $q$ is minimal.  Denote this diameter
by $l_A(p,q)$.
Then, we have
\begin{eqnarray*}
\epsilon + \delta/2 & = & |e_t - q|\\
& \leq & |e_t - p|\\
& \leq & |e_t - p| + l_A(p,q)\\
& < & |e_t - p| + l(p, p')\\
& < & \delta/2 + \epsilon  
\end{eqnarray*}
 Which is a contradiction, and the Lemma is proved.
\end{proof}
 
 It then follows that there is a rational polygonal arc $P$ from $e_t$ to $e_{t+1}$
 that lies in the interior of $J$, does not cross $Q_{t+1}$, and that does not 
 go further than $2^{-t} + 2^{-h(t)-1} \leq 2^{-t + 1}$ from $e_t$.  It now follows
 that a rational polygonal arc with properties \ref{prop1} - \ref{prop4}
 exists.
Such a curve can  be discovered through a simple search procedure which, in order
 to ensure (1), also computes $s$ such that $Q_{t+1} \subseteq \C - \overline{D_{2^{-s}}(P_s)}$.
 
 Let $Q =\bigcup_{t \in \N} Q_t$.  It follows that one endpoint of $Q$ is $z_0$ and 
 the other is $\zeta_0$.  By \ref{prop1} and \ref{prop2}, $Q \cap J = \{\zeta_0\}$.
 
 We have until now thought of each $Q_t$ as a set.   But, we wish 
 to think of $Q$ as a function.  Furthermore, we wish to compute 
 a name of a parameterization of $Q$ from the given data.  To this end, 
 we now backtrack a little and describe in more detail how we parameterize
 each $Q_t$.  Let $v_{t,1}, \ldots, v_{t, n(t)}$ denote the vertices of 
 $Q_t$ in the order in which they are traversed so that $z_0 = v_{t,1}$ and 
 $e_t = v_{t, n(t)}$.  Set 
 \[
 a_{t, j} = \frac{j-1}{n(t)}\mbox{\hspace{0.5in}} j = 1, \ldots, n(t).
 \] 
 Define $Q_t$ to be the function on $[0,1]$ that linearly maps each interval 
 $[a_{t,j}, a_{t, j+1}]$ onto the line segment from $v_{t,j}$ to $v_{t, j+1}$, and that maps 
 all of $[a_{t, n(t)}, 1]$ to the point $v_{t, n(t)}$.  
 
 It follows that $\lim_{t \rightarrow \infty} Q_t(x)$ exists for each $x \in [0,1]$, 
 and we define $Q(x)$ to be the value of this limit.  
 Note that 
 \[
 \parallel Q_t - Q_{t+1} \parallel_\infty \leq 2^{-t} + 2^{-h(t)} \leq 2^{-t + 1}.
 \]
 It now follows that whenever $t' \geq t$, 
 \begin{eqnarray*}
 \parallel Q_t - Q_{t'} \parallel_\infty & \leq & \sum_{j = t}^\infty 2^{-j + 1}\\
 & = & 2^{-t}.
 \end{eqnarray*}
 So, by Theorem 6.2.2.2 of \cite{Weihrauch.2000}, $Q$ is computable uniformly from the given data.
 
 Let $A$ be the circle with center $0$ and radius $1/2$.  
Let $A' = \phi^{-1}[A]$.  Since $A \subseteq \D$, we can compute names of $A'$ and its complement.  We can then compute $R > 0$ such that the circle
of radius $R$ centered at $\zeta_0$ does not intersect $A'$.

We now describe how we compute a name of $\phi(\zeta_0)$.  
Fix $r > 0$ such that $r < R, 1$.   Compute a point $w_1$ on $Q$ besides 
$\zeta_0$ such that every point on $Q$ between $w_1$ and $\zeta_0$ inclusive lies
in $D_r(\zeta_0)$.  In fact, we can take $w_1$ to be a vertex of $Q$.  
We can then compute $m$ such that 
\[
m^2 > \frac{(2\pi)^2}{\log(R) - \log(r)}.
\]
Note that the right side of this inequality approaches $0$ as 
$r$ approaches $0$ from the right. 
Let $T_2$ be the line $x = Re(\phi(w_1))$.  
Compute $\alpha$ such that $\phi(w_1) = |\phi(w_1)|e^{i \alpha}$.  We will proceed under
the assumption that $T_2$ hits $A$.  For, we can apply the following construction and 
argument to $\psi = e^{-i \alpha} \phi$.  Result will be an upper bound on 
$|\psi(w_1) - \psi(\zeta_0)| = |\phi(w_1) - \phi(\zeta_0)|$.  
Let $T_1$ be the line
$x = Re(\phi(w_1)) + m$, and let $T_3$ be the line $x = Re(\phi(w_1)) - m$.
We can thus assume $m$ is small enough 
so that the lines $T_1, T_3$ both 
intersect $A$.  Let $p_1$ be a point where one of these
lines intersects $\partial \D$ and $\overline{p_1\phi(w_1)}$
does not hit $A$.  Let $M = |p_1 - \phi(w_1)|$.  We claim that 
$|\phi(w_1) - \phi(\zeta_0)| < 2M$.  For, suppose otherwise.
Then, when $\phi$ is applied to $\overline{w_1\zeta_0}$, the resulting
arc hits $T_2$ and one of $T_1, T_3$.  Without loss of generality, suppose
it hits $T_1$.  Let $S_k = \phi^{-1}[T_k]$ for $k = 1,2$.  
Hence, $S_1, S_2$ hit $\overline{w_1\zeta_0}$.   Hence, every circle
centered at $\zeta_0$ and whose radius is between $r$ and $R$ inclusive
hits $S_1$ and $S_2$.  This puts us in position to use the ``Length-Area Trick" as follows.  Let $C_{r'}$ denote the circle with radius 
$r'$ centered at $\zeta_0$.  Fix $r \leq r' \leq R$.  The curves 
$S_1, S_2$ have positive minimum distance from each other.
It follows that there are points $z_{1,r'}$ and $z_{2,r'}$ on 
$C_{r'}$ that belong to $S_1, S_2$ respectively and such that 
no points on these curves appear on $C_{r'}$ between 
$z_{1,r'}$ and $z_{2, r'}$.  (We do not claim that we can compute
such points; this is not necessary to prove that our estimate is correct.)
Let $K_{r'}$ denote the arc on $C_{r'}$ from $z_{1,r'}$ to 
$z_{2,r'}$.
Hence, for some $\theta_{1,r'}, \theta_{2,r'}$
\begin{eqnarray*}
|\phi(z_{1,r'}) - \phi(z_{2,r'})| & = & \left| \int_{K_{r'}} \phi'(z) dz \right|\\
& \leq & \int_{\theta_{1,r'}}^{\theta_{2,r'}} |\phi'(z)| r' d\theta\\
\end{eqnarray*}
On the other hand, $\phi(z_{1,r'})$ is on $T_1$ and 
$\phi(z_{2,r'})$ is on $T_2$.  Hence, 
\[
m \leq \int_{\theta_{1,r'}}^{\theta_{2,r'}} |\phi'(z)| r' d\theta.
\]
At the same time, by the Schwarz Integral Inequality (see \emph{e.g.} Theorem 3.5, page 63 of \cite{Rudin.1987}),
\[
m^2 \leq \int_{\theta_{1,r'}}^{\theta_{2,r'}} |\phi'(z)|^2 d\theta \int_{\theta_{1,r'}}^{\theta_{2,r'}} (r')^2 d\theta.
\]
Whence
\begin{eqnarray*}
\frac{m^2}{r'} & \leq & r' \int_{\theta_{1,r'}}^{\theta_{2,r'}} |\phi'(z)|^2 d\theta \int_{\theta_{1,r'}}^{\theta_{2,r'}} d\theta\\
& \leq & 2\pi r' \int_{\theta_{1,r'}}^{\theta_{2,r'}} |\phi'(z)|^2 d\theta.
\end{eqnarray*}
If we now integrate both sides of this inequality with respect to $r'$
from $r$ to $R$, we obtain
\[
m^2[\log(R) - \log(r)] \leq 2\pi \int_r^R \int_{\theta_{1,r'}}^{\theta_{2,r'}} |\phi'(z)|^2r' d\theta dr'.
\]
It follows from the Lusin Area Integral (see \emph{e.g.} Lemma 13.1.2, page 386, of \cite{Greene.Krantz.2002}) that this double integral is no larger than $2\pi$.  Hence, 
\[
m^2 \leq \frac{(2\pi)^2}{\log(R) - \log(r)}.
\]
This is a contradiction.  So, $|\phi(\zeta_0) - \phi(w_1)| < 2M$.  

As $r$ approaches $0$ from the right, $\phi(w_1)$ approaches 
$\phi(\zeta_0)$ and so $m, M$ can be chosen so as to approach $0$.  It follows that we 
can now generate a name of $\phi(\zeta_0)$.
  
We have now computed $\phi$ on $\partial D$.  Since $\phi$ is 
injective, for each $z_0 \in \partial \D$, $\phi - z_0$ has a unique
$0$ on $\partial \D$.  It follows from Corollary 6.3.5 of \cite{Weihrauch.2000} that we can compute $\phi^{-1}$
on $\partial \D$.  It now follows as noted in the introduction to this 
proof that we can compute $\phi^{-1}$ on $\overline{\D}$.
For each $z_0 \in \overline{D}$, $\phi^{-1} - z_0$ has a unique
zero.  Hence, we can compute $\phi$ on $\overline{D}$.
\end{proof}

In \cite{Hertling.1999}, P. Hertling showed that to compute a conformal map $\phi$ of a proper, simply connected domain $D$ onto $\D$, it is necessary and 
sufficient to have a name of $D$ and a name of the boundary of $D$ as a closed
set.  This raises the question as to whether a name of the boundary of $D$ as a closed
set is, when $D$ is Jordan, sufficient to compute the Carath\'eodory extension of $D$.
In fact, this is not the case.  For, from a Carath\'eodory extension of $\phi$, it is possible to 
uniformly compute a parameterization of the boundary of $D$.  But, it follows from
Example 5.1 of \cite{Miller.2002} that there is a Jordan curve that is computable as a compact
set but has no computable parameterization.

\section*{Acknowledgements}

The author thanks his colleagues Dr. Valentin Andreev and Dr. Dale Daniel for helpful
conversations and his wife Susan for support.

\bibliographystyle{spmpsci}      

\end{document}